\documentclass[12pt]{article}
\usepackage{amsmath,amsfonts,amssymb,amsthm,amscd}

\usepackage[all]{xy} 
\usepackage{graphicx}

\title{Connected components of definable groups, and $o$-minimality II.}
\date{June 3, 2011}
\author{Annalisa Conversano\\University of Konstanz \and Anand Pillay \thanks{Supported
by EPSRC grant EP/I002294/1}\\University of Leeds}

\newtheorem{Theorem}{Theorem}[section]
\newtheorem{Proposition}[Theorem]{Proposition}
\newtheorem{Definition}[Theorem]{Definition} 
\newtheorem{Remark}[Theorem]{Remark}
\newtheorem{Lemma}[Theorem]{Lemma}
\newtheorem{Corollary}[Theorem]{Corollary}
\newtheorem{Fact}[Theorem]{Fact}

\newcommand{\R}{\mathbb R}  
  
\newcommand{\Z}{\mathbb Z}  
\newcommand{\N}{\mathbb N}  

\newcommand{\C}{\mathbb C}

\renewcommand{\geq}{\geqslant}

\DeclareMathOperator{\SL}{SL} 
\DeclareMathOperator{\SO}{SO}
\DeclareMathOperator{\SU}{SU}
\DeclareMathOperator{\GL}{GL}

\begin{document}
\maketitle

\begin{abstract} In this sequel to \cite{CCI} we try to give a comprehensive account of the ``connected components" $G^{00}$ and $G^{000}$ as well as the various quotients $G/G^{00}$, $G/G^{000}$, $G^{00}/G^{000}$, for $G$ a group definable in a (saturated) $o$-minimal expansion of a real closed field. Key themes are the structure of $G^{00}/G^{000}$ and the problem of ``exactness" of the $G \mapsto G^{00}$ functor. We prove that the examples produced in \cite{CCI} are typical, and that for any $G$, $G^{00}/G^{000}$ is naturally the quotient of a connected compact commutative Lie group by a dense finitely generated subgroup (where we allow the trivial Lie group).
\end{abstract}
\section{Introduction and preliminaries}
In this paper we restrict ourselves to complete theories $T$ which are $o$-minimal expansions of $RCF$. $G$ will be a group definable in a saturated model ${\bar M}$ of $T$. There are, in principle, three kinds of ``connected component" of $G$; $G^{0} \geq G^{00} \geq G^{000}$.  There are many examples where the first inclusion is strict (such as any definably compact group). In \cite{CCI} we gave the first example where the second inclusion is strict (Example 2.10 and Theorem 3.3 there). Here we will, among other things, extend the analysis of that example to an analysis of the general situation, showing that, in a suitable sense, any nontriviality of $G^{00}/G^{000}$ is explained by the encoding of universal covers of suitable simple noncompact Lie groups as $\bigvee$-definable subgroups of definable central extensions. 

Below (Proposition 1.2) we will give a structure theorem for arbitrary $G$, identifying a key definable section of $G$ which we call $D$, a central extension of a ``strictly non definably compact" semisimple group by a definably compact group. In section 2, we begin the analysis of $G^{00}$ and $G^{000}$ for arbitrary $G$, proving among other things that $G^{00}/G^{000} = D^{00}/D^{000}$.  In section 3 we study the ``connected components" and respective quotients of $D$, proving that $D^{00}/D^{000}$ is naturally the quotient of a commutative compact Lie group by a countable dense subgroup. In section 4 we study various forms of exactness of the functor taking $G$ to $G^{00}$, and again show that the obstacles to exactness come from the group $D$. In particular we show that $G$ has a certain ``almost exactness" property if and only if $G^{00} = G^{000}$. In section 5 we return to the study of $D$, improving results from section 3,  and relating $D^{00}/D^{000}$ to universal covers: in particular proving that $D^{00}/D^{000}$ is naturally isomorphic to the quotient of a certain connected compact commutative Lie group $A$ by a dense finitely generated subgroup $\Lambda$, where $\Lambda$ is a quotient of the fundamental group of a (finite centre) semisimple Lie group related to $D$. 
We will also touch on maximal compact subgroups of Lie groups (section 2) as well as definable amenability (section 4).

As in the prequel, many themes and results of the current paper appear in one form or another in the first author's doctoral thesis \cite{Conversano-thesis} (although there is no explicit mention of $G^{000}$ in the latter). She would like to again thank Alessandro Berarducci for his supervision of her doctoral work, as well as Ya'acov Peterzil for helpful conversations. Some of the writing of this paper was done when the second author was visiting the University of Wroclaw in April 2011, and he would like to thank Wroclaw University for providing excellent conditions, as well as Krzysztof Krupinski for stimulating conversations on the topic of the paper. We should also say that in addition to being a sequel to \cite{CCI} this paper is in many ways a natural continuation of \cite{HPP}.

We will assume acquaintance with \cite{CCI}, and the reader is referred there for the basic definitions.

We work in a saturated model ${\bar M}$, an $o$-minimal expansion of a real closed field, and $G$ denotes a group definably in ${\bar M}$, usually definably connected, i.e. $G = G^{0}$. 

Throughout this paper we will make heavy use of the following result, essentially Corollary 5.3 from \cite{HPP}:
\begin{Lemma} \label{commutator} Suppose $G$ is (definably) a definably connected central extension of a semisimple group. Then $[G,G] \cap Z(G)$ contains no infinite definable set. In particular, for each $n$, $[G,G]_{n}\cap Z(G)$ is finite (where $[G,G]_{n}$ is the set of $n$-fold products of commutators) and $[G,G]\cap Z(G)$ is countable. 
\end{Lemma}
\begin{proof}  Corollary 5.3 of \cite{HPP} says that $[G,G]_{n}\cap Z(G)$ is finite for all $n$. Countability of $[G,G]\cap Z(G)$ follows. But the first part of the Lemma also follows because by compactness any definable subset of $[G,G]\cap Z(G)$ is contained in some $[G,G]_{n}\cap Z(G)$. 
\end{proof}

We now recall and refine the rudimentary structure theorem Proposition 2.6 of \cite{CCI} and fix notation for the remainder of the paper.

As in \cite{CCI} we let $W$ denote the maximal normal ``torsion-free" definable subgroup of $G$. $W$ is definably connected and contained in the {\em solvable radical} $R$ of $G$ (maximal normal definable solvable subgroup). $G/R$ is semisimple.
Proposition 2.6 of \cite{CCI} gives a definably connected definably compact normal subgroup $C$ of $G/W$ such that $(G/W)/C$ is (definably connected) semisimple with no definably compact parts. Let $D_{2}$ denote $(G/W)/C$. Moreover $Z(C)^{0} = Z(G/W)^{0} = R/W$.



\begin{Proposition} \label{decomposition}
Let 
$\bar{G} = G/W$, and $C^{\prime}$ denote the commutator subgroup of $C$. Then 

\begin{enumerate}
\item[$(a)$] $C^{\prime}$ is definable, definably connected, semisimple and definably compact.

\item[$(b)$] $\bar G$ is the almost direct product of $C^{\prime}$ and some definable, definably connected $D$ where,

\item[$(c)$] $D$ is a central extension of a semisimple group with no definably compact parts by a definably compact group.
\end{enumerate}
\end{Proposition}
\begin{proof} By \cite{HPP}, $C = Z(C)^{0}\cdot C^{\prime}$  (almost direct product), where $C^{\prime}$ is definable, definably connected, and semisimple. Clearly both $Z(C)^{0}$ and $C^{\prime}$ are still definably compact, and as remarked above $Z(C)^{0}$ is the connected component of $Z({\bar G})$. Then  ${\bar G}/Z(C)^{0}$ is an almost direct product of $D_{2}$ (which remember was $(G/W)/C)$) and a semisimple definably compact $C_{2}$, such that $C^{\prime}$ is isogenous to $C_{2}$ under the (map induced by the) surjective homomorphism $\pi:{\bar G} \to {\bar G}/Z(C)^{0}$. 
Let $D = \pi^{-1}(D_{2})$. Then $D$ is a definable, definably connected subgroup of ${\bar G}$, and moreover $D$ is a central extension of $D_{2}$ (a semisimple group with no definably compact part) by the definably compact group $Z(C)^{0}$. Clearly ${\bar G}$ is an almost direct product of $C^{\prime}$ and $D$.
\end{proof}

We recall from \cite{CCI}:

\begin{Definition} \label{goodecomposition}
We say that $G$ has a {\em good decomposition}, if the exact sequence $1 \to C \to G/W \to D_{2}$ almost splits, namely $G/W$ can be written as an almost direct product  $C\cdot D_{3}$ for some definable, definably connected, subgroup $D_{3}$ of $G/W$ which is semisimple with no definably compact parts.
\end{Definition}
 
\begin{Remark} \label{goodecompositionD} 
$G$ has a good decomposition if and only if the commutator subgroup $D^{\prime}$ of $D$ is definable.
\end{Remark}

\begin{proof}
If the commutator subgroup $D^{\prime}$ of $D$ is definable, then by Lemma 1.1, its intersection with $Z(D)$ is finite. On the other hand  by 3.1 (v) of \cite{HPP}, $D/Z(D)$ is perfect. It follows that $D$ is an almost direct product  of $Z(D)^{0} = Z(C)^{0}$ with the semisimple definable group $D^{\prime}$  (with no definably compact parts). But then $G/W = C^{\prime}\cdot Z(C)^{0}\cdot D^{\prime} = C \cdot D^{\prime}$ as required.

Conversely, suppose $G$ has a good decomposition: ${\bar G} = G/W = C\cdot D_{3} = C^{\prime}\cdot Z(C)^{0}\cdot D_{3}$. Then $Z(C)^{0}\cdot D_{3}$ is (definably) isogenous to $D$. But (using perfectness of the semisimple group $D_{3}$) the commutator subgroup of $Z(C)^{0}\cdot D_{3}$ is precisely $D_{3}$ which is definable. So the commutator subgroup of $D$ is also definable.

\end{proof}

\section{$G^{00}$ and  $G^{000}$}  
We recall that $G^{00}$ is the smallest type-definable (over a small set of parameters) subgroup of $G$ of bounded index, and $G^{000}$ is the smallest bounded index subgroup of $G$ which is $Aut({\bar M}/A)$-invariant for some small set $A$ of parameters. We refer to the prequel \cite{CCI} for more details and background. In any case we have $G \geq G^{00} \geq G^{000}$ and both $G^{00}$, $G^{000}$ are normal in $G$. 
In this section we begin the analysis of these ``connected components" in the light of the structure theorems, 2.6 of \cite{CCI} and 1.2 above.

We start with a couple of easy results valid for arbitrary theories. So $T$ is an arbitrary complete theory, ${\bar M}$ a saturated model, and $G, H,..$  definable groups in ${\bar M}$. We will also use notation such as $G_{A}^{0}$, $G_{A}^{00}$, $G_{A}^{000}$ for $A$ a small set of parameters of which $G$ is defined, where for example $G_{A}^{00}$ denotes the smallest type-definable over $A$ subgroup of $G$ of bounded index. In general $G_{A}^{00}$ may depend on $A$, but in the $o$-minimal context (more generally in any theory with $NIP$), these connected components do not depend on $A$ (i.e. remain unchanged as $A$ grows). 

\begin{Lemma} \label{image}
Let $G, H$ be $\emptyset$-definable groups, and $f:G\to H$ a $\emptyset$-definable surjective homomorphism. Then
\begin{enumerate}

\item[$(i)$] $f(G^{0}_{\emptyset}) = H_{\emptyset}^{0}$, $f(G_{\emptyset}^{00}) =  H_{\emptyset}^{00}$, and $f(G_{\emptyset}^{000}) =
H_{\emptyset}^{000}$.

\item[$(ii)$] Suppose moreover that $\ker(f)$ is finite, $G = G_{\emptyset}^{0}$, and $H = H_{\emptyset}^{00} = H_{\emptyset}^{000}$. Then $G = G_{\emptyset}^{00} = G_{\emptyset}^{000}$.

\end{enumerate}
\end{Lemma}

\begin{proof}
$(i)$ is clear. 
\newline
$(ii)$ We first show that $G_{\emptyset}^{00} = G$. By $(i)$ and our hypothesis, $f(G_{\emptyset}^{00}) = H$ so as $\ker(f)$ is finite, $G_{\emptyset}^{00}$ has finite index in $G$, but then $G_{\emptyset}^{00}$ is $\emptyset$-definable, so equals $G_{\emptyset}^{0}$ which by hypothesis equals $G$. 
\newline
Likewise $G_{\emptyset}^{000}$ has finite index in $G = G_{\emptyset}^{00}$, so by Lemma 3.9 (1) of \cite{Gismatullin-Newelski} we get equality. 
\end{proof}

\begin{Lemma} \label{product}
Suppose that $G$ is $\emptyset$-definable, and $H, K$ are $\emptyset$-definable subgroups of $G$ which commute with each other. Then 

\[
G_{\emptyset}^{0} = H_{\emptyset}^{0}\cdot K_{\emptyset}^{0} \qquad
G_{\emptyset}^{00} = H_{\emptyset}^{00}\cdot K_{\emptyset}^{00} \qquad  
G_{\emptyset}^{000} = H_{\emptyset}^{000}\cdot K_{\emptyset}^{000}. 
\] 
\end{Lemma}

\begin{proof} Again this is straightforward. We briefly describe the $000$ case.
\newline
First we show that $(H\times K)_{\emptyset}^{000} = H_{\emptyset}^{000}\times K_{\emptyset}^{000}$. Note that the right hand side is $Aut({\bar M})$-invariant of bounded index in $H\times K$ so contains the left hand side. But the intersection of the left hand side with $H$ has bounded index in $H$ and is invariant, so {\em contains}  $H_{\emptyset}^{00}$ so, by the previous sentence equals $H_{\emptyset}^{00}$. Likewise the left hand side intersect $K$ equals $K_{\emptyset}^{000}$. This shows the required equality.  
\newline
Now $G$ is the image of $H\times K$ under a canonical $\emptyset$-definable surjective homomorphism $f$. By Lemma \ref{image} and the first paragraph, 
$G_{\emptyset}^{000} = 
f(H_{\emptyset}^{000}\times K_{\emptyset}^{000})$ which clearly equals the internal product  $H_{\emptyset}^{000}\cdot K_{\emptyset}^{000}$. 
\end{proof}

\vspace{5mm}
\noindent
We now specialize to the case where $T$ is an $o$-minimal expansion of $RCF$, so ${\bar M} = (K,+,\cdot,..)$ for a saturated real closed field $K$. We just talk about $G^{0}, G^{00}, G^{000}$ and need not worry about parameters. 

\begin{Lemma} \label{compact} 
Suppose $G$ is definably compact. Then $G^{00} = G^{000}$.
\end{Lemma}

\begin{proof} This follows from Lemma 5.6 of \cite{NIPII} and the paragraph before it. Namely the quoted lemma says that if $G$ has a global ``$f$-generic type" then $G^{00} = G^{000}$, and the paragraph before it says that a definably compact group (in an $o$-minimal theory) has a global $f$-generic type. 
\end{proof}

\begin{Lemma} \label{torsion-free} 
Suppose $G$ is torsion-free. Then $G = G^{00} = G^{000}$. 
\end{Lemma}

\begin{proof} This is proved by induction on the dimension of $G$. 

Suppose $\dim(G) = 1$. From results in \cite{Razenj}, we may assume that $G$ is an open interval in  $1$-space with continuous group operation. The global type at ``$+\infty$" is both definable and $G$-invariant. In particular $p$ is $f$-generic, so we can again apply Lemma 5.6 of \cite{NIPII}.

Suppose now $\dim(G) > 1$.
By \cite{PeSta} there is a normal definable
subgroup  $H$ of $G$ such that $G/H$ is $1$-dimensional (and torsion-free), and the induction hypothesis applies to $H$. So $H = H^{00} = H^{000}$, whereby each of $G^{00}$ and $G^{000}$ contain $H$. By Lemma \ref{image}$(i)$ and the $1$-dimensional case, we see that $G = G^{00} = G^{000}$.
\end{proof}

\begin{Lemma} \label{almsimple}
Suppose $G$ is definably connected and definably almost simple. If $G$ is not definably compact then $G = G^{00} = G^{000}$.
\end{Lemma}

\begin{proof} Note that $Z(G)$ is finite, and $G/Z(G)$ is definably simple and not definably compact. By Corollary 6.3 of \cite{PPSIII}, 
$G/Z(G)$ is simple as an abstract group. In particular $G/Z(G) = (G/Z(G))^{00} = (G/Z(G))^{000}$, as the latter are normal nontrivial subgroups. By \ref{image}$(ii)$, $G = G^{00} = G^{000}$.
\end{proof} 


We can now prove:

\begin{Proposition} \label{goodg2=g3} 
Let $G$ be a definable, definably connected group  (in a saturated model of an $o$-minimal expansion of a real closed field). Assume that $G$ has a good decomposition in the sense of Definition \ref{goodecomposition}. Then $G^{00} = G^{000}$.
\end{Proposition}

\begin{proof} Let $W$, $C$, $D$  be the constituent elements of the decomposition of $G$ given by Proposition \ref{decomposition}, and $D^{\prime} = [D, D]$ the commutator subgroup of $D$. As we are assuming that $G$ has 
a good decomposition, $D^{\prime}$ is definable (\ref{goodecompositionD}) and we have an exact sequence
\[
1\ \to\ W\ \to\ G\  \to\ C\ \cdot D^{\prime}\ \to\ 1
\]

\noindent
where $W$ is torsion-free, $C$ is definably compact, $D^{\prime}$ is semisimple with no definably compact parts, and so $C \cdot D^{\prime}$ is an almost direct product. (And all groups mentioned are definable, definably connected). Let us denote $C \cdot D^{\prime}$ by $H$.

By Lemma \ref{torsion-free}, $W = W^{00} = W^{000}$. So both $G^{00}$ and $G^{000}$ contain $W$. Hence 
\[
(*) \qquad \qquad G^{00} = \pi^{-1}(\pi(G^{00})) \quad \mbox{and} \quad G^{000} = \pi^{-1}(\pi(G^{000})). 
\]

\noindent
But by Lemma \ref{image}, 
\[
(**) \qquad \qquad \pi(G^{00}) = H^{00} \quad \mbox{and} \quad \pi(G^{000}) = H^{000}. 
\]
  
\noindent 
Moreover by Lemmas \ref{compact}, \ref{almsimple}, \ref{image}, $H^{00} = H^{000}$, so by $(*)$ and $(**)$, $G^{00} = G^{000}$. 
\end{proof}

\vspace{2mm}
\noindent
Note of course that Proposition \ref{goodg2=g3} holds without the definably connected assumption on $G$, as for any $G$, $G^{00} = (G^{0})^{00}$ and $G^{000} = (G^{0})^{000}$. Note also that by Remark 2.9 of \cite{CCI}, we conclude that $G^{00} = G^{000}$ whenever $G$ is either linear (i.e. a definable subgroup of some $\GL_{n}(K)$) or   algebraic (i.e. of the form $H(K)$ for $H$ an algebraic group over $K$). \\

The following refinement of \ref{goodg2=g3} reduces to the group  $D$ from \ref{decomposition}. We will systematically use the notation from Proposition 1.2 in the next few results:  

\begin{Proposition} \label{g2andg3iffd2andd3}
Let $G$ be a definable group, and $D$ as in Proposition \ref{decomposition}. 
Then $G^{00}/G^{000}$ is isomorphic to $D^{00}/D^{000}$. In particular $G^{00} = G^{000}$ if and only if $D^{00} = D^{000}$. 
\end{Proposition} 

\begin{proof}
Let $W$, $C^{\prime}$, $D$ be as in Proposition \ref{decomposition}. 
Set $\bar{G} = G/W$. 

By Lemmas
\ref{torsion-free} and \ref{image} the quotient map induces isomorphisms between $G/G^{00}$ and ${\bar G}/{\bar G}^{00}$ and between $G/G^{000}$ and 
${\bar G}/{\bar G}^{000}$. So we may work with ${\bar G}$ in place of $G$. By  \ref{product}, ${\bar G}^{00} = (C^{\prime})^{00}\cdot D^{00}$ and ${\bar G}^{000} = (C^{\prime})^{000}\cdot D^{000}$. By Lemma \ref{compact} $(C^{\prime})^{00} = (C^{\prime})^{000}$. Hence 
${\bar G}^{00}/{\bar G}^{000}$ equals $D^{00}/D^{000}$.
\end{proof}

We will give a reasonably complete description of $D^{00}/D^{000}$ in section 3 which will be  elaborated on in section 5.

For the rest of this section we will make some further observations about the various quotients in the general case. Note to begin with that any map between bounded hyperdefinable sets induced by a type-definable map has to be continuous with respect to the relevant ``logic topologies".
\begin{Proposition} \label{goverg2}
(i) $G/G^{00}$ is isomorphic (as a compact topological group) to $C/({\bar G}^{00} \cap C)$.
\newline
(ii) $G/G^{000}$ is isomorphic to  $C/({\bar G}^{000}\cap C)$.
\end{Proposition}
\begin{proof}
(i) We already remarked in the proof of Proposition 2.7 that $G/G^{00}$ is isomorphic to ${\bar G}/{\bar G}^{00}$. 
We have the exact sequence $1 \to C \to {\bar G} \to_{\pi} D_{2}$ where $D_{2}$ is semisimple (with no definably compact part). By Lemma 2.1 and Lemma 2.5, $\pi({\bar G}^{00}) = D$.  It follows that quotienting by ${\bar G}^{00}$ induces an isomorphism between $C/({\bar G}^{00} \cap C)$ and $G/G^{00}$.
\newline
The same argument yields (ii). 

\end{proof}

Recall from \cite{NIPI} that $G$ is said to have {\em very good reduction} if there is a sublanguage $L_{0}$ of the language $L$ of $T$ and an elementary substructure $M_{0}$ of ${\bar M}|L_{0}$ such that the underlying set of $M_{0}$ is $\R$ and $G$ is definable over $M_{0}$ in the reduct 
${\bar M}|L_{0}$. In this case we can speak of $G(\R)$ which will be a real Lie group. When $G$ is definably compact (and so $G(\R)$ is compact) we know from \cite{NIPI} and \cite{Peterzil-Pillay}  that $G/G^{00}$ is isomorphic (as a compact topological group) to  $G(\R)$, via the standard part map. One can ask what happens in the general case. 
In Theorem 6 of \cite{Iwasawa} the existence of {\em maximal compact subgroups} of any connected Lie group is proved,  as well as the connectedness and conjugacy of these maximal compacts, and a further decomposition theorem.  What is especially relevant to our considerations is:

\begin{Fact} Let $G$ be a connected Lie group, and $N$ a closed normal solvable subgroup such that
\newline
(i) $G/N$ is compact,
\newline
(ii) There are closed $\{1\} = N_{0}< N_{1} < N_{2} < ... < N_{k} = N$ such that  for $i<k$, $N_{i}$ is normal in $N_{i+1}$ and the quotient $N_{i+1}/N_{i}$ is Lie isomorphic to $(\R,+)$.
\newline
Then there is a (maximal) compact subgroup $C_{1}$ of $G$, such that $C_{1}\cap N = \{1\}$ and every element of $G$ can be (uniquely) written as a product of an element of $C_{1}$ and an element of $N$.
\end{Fact}
\begin{proof}  This is contained in Lemma 3.7 of \cite{Iwasawa} in the case where $N = N_{1}$. The fact then follows easily by induction. 
\end{proof}

\begin{Proposition} Suppose that $G$ has very good reduction and ${\bar G}$ is definably compact (i.e. equals $C$). Then $G/G^{00}$ is isomorphic to a maximal compact subgroup of $G(\R)$.
\end{Proposition} 
\begin{proof} We clearly have that $1 \to  W \to G \to C \to 1$ is definable in the reduct ${\bar M}|L_{0}$ over $\R$. 
By Proposition 2.8 $G/G^{00}$ is isomorphic to $C/C^{00}$ which as remarked is isomorphic to  $C(\R)$ as a compact group.  On the other hand we have the exact sequence of Lie groups:
\[
1 \to W(\R) \to G(\R) \to C(\R) \to 1. 
\]

Clearly $W(\R)$ satisfies the hypothesis on $N$ in Fact 2.9, and $C(\R)$ is compact. So by Fact 2.9, there is a compact subgroup $C_{1}$ of $G(\R)$ such that every element of $G(\R)$ can be written uniquely as a product of an element of $C_{1}$ and an element of $W(\R)$. Hence quotienting by $W(\R)$ yields an isomorphism (of Lie groups) between $C_{1}$ and $C(\R)$ which completes the proof.
\end{proof}

Note that the converse is also true: If $G$ (with very good reduction) is not definably compact then a maximal compact subgroup of $G(\R)$ will have dimension 
$> \dim(G/G^{00})$ so could not be isomorphic to $G/G^{00}$. Because, if $D_{2}$ is the semisimple with no definable compact parts, part of $G$ then the semisimple Lie group $D_{2}(\R)$ has a maximal compact subgroup, say $C_{2}$, of positive dimension. This will lift to a maximal compact $C_{3}$ of $G(\R)$ containing $C_{1}$ (where $C_{1}$ is as in the proof above). But then $\dim(C_{3}) > \dim(C_{1}) = \dim(C/C^{00}) \geq \dim(G/G^{00})$.

\section{$D^{00}$ and $D^{000}$}
This section is devoted to the investigation of the group $D$, a definable central extension of a semisimple group with no definably compact parts by a definably compact group. In \cite{CCI} we gave examples of such groups where $D^{00} \neq D^{000}$. We will show here that $D^{00}/D^{000}$ is  abelian, 
and is moreover (naturally) isomorphic to a quotient of a commutative compact Lie group by a countable dense subgroup. This will be improved on in various ways  in section 5 where we make use of properties of universal covers of simple Lie groups.

We let $D^{\prime}$ denote the commutator (derived) subgroup of $D$, and $(D^{\prime})^{\prime}$ the second derived subgroup, namely the commutator
subgroup of $D^{\prime}$. We introduce some new notation by letting $\Gamma$ denote the connected component of 
$Z(D)$ . We have the (definable) exact sequence $$1 \to \Gamma \to D \to_{\pi} D_{2} \to 1$$ where $D_{2}$ is  definable, connected and semisimple (we allow a finite centre) with no definably compact parts.  And $\Gamma$ is definably compact (commutative). We will prove:

\begin{Theorem} \label{g2andg3}

\begin{enumerate}

\item $D^{000} = \Gamma^{00} \cdot (D^{\prime})^{\prime}$.

\item $D^{00} = E \cdot (D^{\prime})^{\prime}$, where $E$ is the smallest type-definable subgroup of $\Gamma$ containing both
$\Gamma^{00}$ and $\Gamma \cap (D^{\prime})^{\prime}$. 
\end{enumerate}
\end{Theorem} 
\begin{proof}
1.  We make a couple of claims:
\newline
{\em Claim I.}  $D^{000}$ contains $\Gamma^{00}$.
\newline
This is because $D^{000}\cap\Gamma$ has bounded index in $\Gamma$ hence contains $\Gamma^{000}$ which equals $\Gamma^{00}$ by Lemma 2.3.

\vspace{2mm}
\noindent
{\em Claim II.}  $D^{000}$ contains  $(D^{\prime})^{\prime}$.
\newline
{\em Proof of Claim II.}  Let $H =  D^{\prime}\cap D^{000}$. By Claim 3.1(v) of \cite{HPP}, $D_{2}$ is perfect, hence $\pi(D^{\prime}) = D_{2}$. Now $H$ is a 
normal subgroup of $D^{\prime}$ of bounded index (in fact index bounded by the index of $D^{000}$ in $D$ which is at most continuum). Hence $\pi(H)$ is a normal subgroup of $D_{2}$ of bounded index. But $D_{2}$ is an almost direct product of finitely many groups, each of which is simple as an abstract group (modulo a possibly finite centre). This implies that $\pi(H) = D_{2} = \pi(D^{\prime})$. Hence as $\Gamma = \ker(\pi)$, we see that $D^{\prime} = (\Gamma \cap D^{\prime})\cdot H$. Hence $(D^{\prime})^{\prime} \subseteq H \subseteq D^{000}$, proving Claim II.

\vspace{2mm}
\noindent
Note again that as $D_{2}$ is perfect  $\pi((D^{\prime})^{\prime}) = D_{2}$.
By Claims I and II, 
\newline
(*) $D^{000}\supseteq\Gamma^{00}\cdot (D^{\prime})^{\prime}$. 
\newline
But as $\Gamma^{00}$ has bounded index in $\Gamma$  and $(D^{\prime})^{\prime}$ projects onto $D_{2}$ we see that $\Gamma^{00}\cdot (D^{\prime})^{\prime}$ has bounded index in $D$ and is clearly $Aut({\bar M}/A)$ invariant where $A$ is a set of parameters over which $D$ is defined. Hence in (*) we obtain equality, yielding 1.

\vspace{2mm}
\noindent
2. We start with
\newline
{\em Claim III.}  $D^{00}$ contains $E \cdot (D^{\prime})^{\prime}$ where $E$ is as in the statement of the theorem.
\newline
{\em Proof of Claim III.} First $D^{00}$ contains $D^{000}$ which contains $(D^{\prime})^{\prime}$ by 1. Secondly $D^{00}\cap\Gamma$ contains 
$\Gamma^{00}$ as well as $D^{000}\cap\Gamma$ and as we have just seen the latter contains $(D^{\prime})^{\prime}\cap \Gamma$. But $D^{00}\cap\Gamma$ is 
also type-definable so contains the smallest type-definable subgroup of $\Gamma$ containing both $\Gamma^{00}$ and $(D^{\prime})^{\prime}\cap \Gamma$, 
which is precisely $E$. Hence $D^{00}$ contains $E \cdot (D^{\prime})^{\prime}$, proving Claim III.

\vspace{2mm}
\noindent
Clearly $E \cdot (D^{\prime})^{\prime}$ has bounded index in $D$  (it contains $D^{000}$ by 1). So bearing in mind Claim III, it suffices, to complete the proof of 2, to prove:

\vspace{2mm}
\noindent
{\em Claim IV.}  $E \cdot (D^{\prime})^{\prime}$  is type-definable. 
\newline
{\em Proof of Claim IV.} By Claim 3.1(v) of \cite{HPP}, $D_{2}$ is perfect and equals $[D_{2},D_{2}]_{n}$ for some $n$, namely every element of $D_{2}$ is a product of at most $n$ commutators. It follows that $\pi(X) = D_{2}$ where  $X = [[D,D]_{n},[D,D]_{n}]_{n}$. So for any 
$a\in (D^{\prime})^{\prime}$ there is $b\in X$ such that $\pi(a) = \pi(b)$. Hence $a = cb$ where $c = ab^{-1}\in \Gamma\cap (D^{\prime})^{\prime}$. Hence $c\in E$. So we have shown that $(D^{\prime})^{\prime}$ is contained in $E\cdot X$ (the set of products of elements of $E$ with elements of $X$). Hence $E \cdot (D^{\prime})^{\prime} = E\cdot X$. As $E$ is type-definable and $X$ is definable, $E\cdot X$ is type-definable. This completes the proof of Claim IV and so also of 2. 
\end{proof}

Following the notation in Theorem 3.1 we conclude:
\begin{Corollary} $D^{00}/D^{000}$ is isomorphic to the quotient of the compact (not necessarily connected) commutative Lie group $E/\Gamma^{00}$ by the countable dense subgroup  $(\Gamma^{00} \cdot(\Gamma \cap (D^{\prime})^{\prime}))/\Gamma^{00}$.
\end{Corollary}
\begin{proof}
We start by giving an explanation. $\Gamma/\Gamma^{00}$ is a (connected) compact commutative Lie group, when equipped with the logic topology (see section 3 of \cite{CCI} for a full discussion of the logic topology on bounded hyperdefinable sets and groups). The closed subgroups of $\Gamma/\Gamma^{00}$ correspond precisely to the type-definable subgroups of $\Gamma$ which contain $\Gamma^{00}$, and $E$ is an example of the latter. Now $(\Gamma^{00} \cdot(\Gamma \cap (D^{\prime})^{\prime}))/\Gamma^{00}$ is a subgroup of $E/\Gamma^{00}$, and its closure in $E/\Gamma^{00}$ is also a subgroup. Hence by the definition of $E$, the closure of 
$(\Gamma^{00} \cdot(\Gamma \cap 
(D^{\prime})^{\prime})/\Gamma^{00}$  in $E/\Gamma^{00}$ is precisely $E/\Gamma^{00}$. Also by Lemma 1.1, 
$\Gamma \cap (D^{\prime})^{\prime}$ is countable, so this proves  the Corollary.
\end{proof}

Let us remark here that $E/\Gamma^{00}$ may not be connected. We give an example in the next section where it is finite. But also note that if $E/\Gamma^{00}$ {\em is} finite then $E = D^{00} = D^{000}$. 

Using Proposition 2.7 we obtain:

\begin{Corollary} For any definable group $G$, $G^{00}/G^{000}$ is isomorphic (as an abstract group) to the quotient of a compact commutative Lie group by a countable dense subgroup. In particular $G^{00}/G^{000}$ is commutative.
\end{Corollary}

\vspace{2mm}
\noindent
We have already remarked that as far as bounded hyperdefinable groups are concerned, our isomorphisms are isomorphisms of topological (hence Lie) groups. One would like to have the isomorphism in the Corollary belong to a more structured category. This depends partly on how we want think of objects such as $G^{00}/G^{000}$. This will be treated in subsequent work, in particular the complexity of $G^{00}/G^{000}$ from the point of view of Borel equivalence relations on Polish spaces. 

\vspace{2mm}
\noindent
A final remark, with above notation:
\begin{Proposition} $D^{00} = D^{000}$ if and only $\Gamma^{00}$ has finite index in $\Gamma\cap D^{000}$. 
\end{Proposition} 
\begin{proof} Suppose first that $D^{00} = D^{000}$, namely that $D^{000}$ is type-definable. By Theorem 3.1 and notation there clearly $\Gamma^{00}\cdot (\Gamma \cap (D^{\prime})^{\prime}) = \Gamma\cap D^{000}$ so is type-definable. Now $\Gamma \cap (D^{\prime})^{\prime}$ is a countable group $A$ say by Lemma 1.1. 
Then $(\Gamma^{00}\cdot A)/\Gamma^{00}$ is a countable closed subgroup of the commutative Lie group $\Gamma/\Gamma^{00}$, so has to be finite. 

Conversely, if $\Gamma^{00}$ has finite index in $\Gamma\cap D^{000}$ then the latter is definable and has to equal $E$. Hence using Theorem 3.1, 
$D^{00} = E\cdot (D^{\prime})^{\prime} \subseteq D^{000}$ so we have equality.

\end{proof}



\section{Exactness} 
 
We now consider the question of the {\em exactness} of the functor which takes a definable, definably connected, group $G$ to $G^{00}$; namely if 
\[
 1\ \to\ L\ \to\ G\ \to\ H\ \to 1 
\]

\noindent
is an exact sequence of definable, definably connected, groups, do we get an exact sequence
$1 \to L^{00} \to G^{00} \to H^{00} \to 1$ ?  We can ask the same question for the $G \mapsto G^{000}$ functor.
By \ref{image} $(i)$ ($G^{00}$ maps onto $H^{00}$, $G^{000}$ maps on to $G^{000}$) exactness of the induced sequences amounts to  $L\cap G^{00} = L^{00}$ and $L\cap G^{000} = H^{000}$.  
When $G$ is definably compact, a positive answer is obtained by Berarducci \cite{Berarducci}, and we are partly motivated by trying to generalize his results.
But note that Theorem 3.3 in \cite{CCI} gives a negative answer in general. There we had the exact sequence 
\[
1\ \to\ \SO_{2}(K)\ \to\ G_{1}\ \to\ \SL_{2}(K)\ \to 1.
\]

\noindent
The analysis (and notation) there gives that $G_{1}^{00} \cap \SO_{2}(K) = \SO_{2}(K)$, and $G_{1}^{000} \cap \SO_{2}(K) =  \SO_{2}(K)^{00} \cdot \Z$.
 In particular $\SO_{2}(K)^{00}$ has {\em infinite} (but of course bounded) index in each of 
 $G_{1}^{00} \cap \SO_{2}(K)$ and $G_{1}^{000} \cap \SO_{2}(K)$.\\

Even when $G$ has a good decomposition (Definition 1.3)  exactness can fail: 
Let $G = R \cdot S$ be the almost direct product of $R \cong \SO_{2}(K)$ with $S \cong \SL_{2}(K)$ obtained by identifying the square roots of the identity in both groups. Then $G^{00} = R^{00}\cdot S^{00}$. But $S^{00} = S$, so its intersection with $R$ contains a finite subgroup, whereas $R^{00}$ is the ``infinitesimals" of $R$ and contains no finite subgroup. In this example we have ``almost exactness" in the sense that $R^{00}$ has 
{\em finite index} in $G^{00}\cap R$. See also Remark 4.7 for a related example.
In any case this motivates the following definitions. 

\begin{Definition}
Let $G$ be a definable group. We say that

\begin{enumerate}

\item $G$ has {\em the almost exactness property} if for every normal definable subgroup of $G$ , $H \lhd G$,
$H^{00}$ has finite index in $G^{00} \cap H$.

\item $G$ has {\em the exactness property} if for every normal definable subgroup of $G$ $H \lhd G$,
$ H^{00} = G^{00} \cap H$.

\item $G$ has {\em the strong exactness property} if for every definable subgroup $H < G$,
$H^{00} = G^{00} \cap H$.

\end{enumerate}
\end{Definition}

There are obvious $G^{000}$ analogues.

\vspace{2mm}
\noindent
We will show the following: 

\begin{itemize}
\item the class of definable groups with the almost exactness property coincides with the class
of definable groups $G$ such that $G^{00} = G^{000}$: \ref{exactnessg2g3};

\item for definable groups $G$ with the exactness property, $G/G^{00}$ is isomorphic to $C/C^{00}$, which follows from 4.8 
(moreover, to justify the definition, we give an easy example of a group with the almost exactness property such that $G/G^{00}$ is not isomorphic to $C/C^{00}$: \ref{SU2SL2});

\item the class of definable groups with the strong exactness property coincides with the class of definably amenable groups: 4.10.
\end{itemize}

\begin{Lemma} \label{reduction-almost-exactness}
Let $H \lhd G$ be definably connected groups such that:

\begin{enumerate}
\item[$(a)$] $H^{00}$ has finite index in $G^{00} \cap H$,

\item[$(b)$] $H$ has the almost exactness property,

\item[$(c)$] $G/H$ has the almost exactness property.
\end{enumerate}

\noindent
Then $G$ has the almost exactness property.
\end{Lemma}

\begin{proof} Let $N \lhd G$ be a definable subgroup. We want
to show that $N^{00}$ has finite index in $G^{00} \cap N$,
i.e. there is $n \in \N$ such that $x^n \in N^{00}$ for each $x \in G^{00} \cap N$.

If $[(G^{00} \cap H) : H^{00}] = n_1$ (by condition $(a)$), $[(H^{00} \cap N) : (H \cap N)^{00}] = n_2$ (by condition $(b)$), 
$[((N \cdot H/H) \cap (G/H)^{00}) : (N \cdot H/H)^{00} ] = n_3$ (by condition $(c)$), we claim that $n = n_1 \cdot n_2 \cdot n_3$ works.  

Assume $x \in G^{00} \cap N$. \\
If $x \in H$, then $x^{n_1} \in H^{00}$. So $x^{n_1} \in H^{00} \cap N$, and $x^{n_1 \cdot n_2} \in (H \cap N)^{00} \subset N^{00}$. \\
If $x  \not \in H$, consider the canonical projection $\pi \colon G \to G/H$ and let $\pi(x) = \bar{x} \in G/H$.
We have $\bar{x} \in \bar{N} = \pi(N)$, $\bar{x} \in \pi(G^{00}) = (G/H)^{00}$, and therefore $\bar{x}^{n_3} \in \bar{N}^{00}$. Then 
\[
x^{n_3} = y \cdot h,
\]

for some $y \in N^{00}$ and $h \in H$. Now:
\begin{align*}
x^{n_3}, y \in G^{00}\ &\Rightarrow\ h \in G^{00}\  \stackrel{(a)}{\Rightarrow}\  h^{n_1} \in H^{00},\\
x, y \in N\ &\Rightarrow\ h \in N, \\
h^{n_1} \in H^{00} \cap N \ &\stackrel{(b)}{\Rightarrow}\ h^{n_1 \cdot n_2} \in (N \cap H)^{00} \subseteq N^{00}.\\
\end{align*}

Therefore $x^n = h^n \cdot y^{\prime}$, for some $y^{\prime} \in N^{00}$ (both $N^{00}$ and $H$ are normal in $G$), and our claim is proved.
\end{proof}
 
\begin{Remark} \label{reduction-exactness}
After replacing condition $(a)$ with $H^{00} = G^{00} \cap H$, one can show corresponding Lemmas for
the exactness property and the strong exactness property.
\end{Remark}
 
\begin{Theorem} \label{exactnessg2g3}
Let $G$ be a definable group. Then $G^{00} = G^{000}$ if and only if $G$ has the almost exactness property.   
\end{Theorem}

To this end, we show first:

\begin{Lemma} \label{reductiontoZ(D)}
Let $G$ be a definably connected group. Then $G$ has the almost exactness property if and only if
$Z(D)^{00}$ has finite index in $Z(D)^{0} \cap D^{00}$.
\end{Lemma}

\begin{proof}
Suppose $G$ has the almost exactness property. In particular, for the solvable radical $R$ of $G$, $R^{00}$
has finite index in $G^{00} \cap R$. 
Let $\pi : G \to G/W$ be the canonical projection. Then $\pi(R^{00}) = Z(D)^{00}$ and $\pi(G^{00} \cap R) = D^{00} \cap Z(D)^0$. Hence $Z(D)^{00}$ has finite index in $Z(D)^{0} \cap D^{00}$.

Conversely, suppose $Z(D)^{00}$ has finite index in $Z(D)^{0} \cap D^{00}$. Then we can apply Lemma \ref{reduction-almost-exactness} (with $H = Z(D)^0$) to deduce that the group $D$ has the almost exactness property. By the same Lemma (with $H = D$), also $\bar{G} = G/W$ has the almost exactness property. Finally $G$ has the almost exactness property taking $H = W$.
\end{proof}

\begin{proof}[Proof of Theorem \ref{exactnessg2g3}] 
Suppose $G^{00} = G^{000}$. By Proposition 2.7, $D^{00} = D^{000}$, so by Proposition 3.4, $Z(D)^{00}$ has finite index in $Z(D)^{0}\cap D^{00}$. By Lemma 4.5, $G$ has the almost exactness property.

Conversely, if  $G$ has the almost exactness property then by the 4.5, $Z(D)^{00}$ has finite index in $Z(D)^{0}\cap D^{00}$, namely $\Gamma^{00}$ has finite index in $\Gamma \cap D^{00}$ in the language of Prop. 3.4. But $D^{000}$ is contained in $D^{00}$ (and contains $\Gamma^{000}$). Hence $\Gamma^{00}$ has finite index in $\Gamma\cap D^{000}$ and we can apply Proposition 3.4 and 2.7 to conclude that $G^{00} = G^{000}$. 

\end{proof}

We conclude using \ref{goodg2=g3}: 
\begin{Corollary}  If $G$ has a good decomposition (for example if $G$ is linear, or $G$ is algebraic) then  $G$ has the almost exactness property. 
\end{Corollary}

Finally we relate the various exactness properties to earlier structure theorems and results as well as definable amenability.

\begin{Remark} \label{almexact-isog}
$G/G^{00}$ is isogenous to $C/C^{00}$ if and only if $G$ has the almost exactness property.
\end{Remark}

\begin{proof}
We can assume $G = \bar{G} = G/W$. 

Suppose first $G/G^{00}$ is isogenous to $C/C^{00}$. Then by Proposition \ref{goverg2}, $C^{00}$ has finite index in $G^{00} \cap C$. So we can apply Lemma \ref{reduction-almost-exactness} (with $H = C$) and conclude that $G$ has the almost
exactness property.
Conversely, if $G$ has the almost exactness property, then in particular $C^{00}$ has finite index in $G^{00} \cap C$. Again by \ref{goverg2} $G/G^{00}$ is isogenous to $C/C^{00}$.
\end{proof}

It is well-known that connected Lie groups have a (unique) maximal normal connected compact subgroup (\cite[Theorem 14]{Iwasawa}). If $G$ has very good reduction,
then it is easy to see that the maximal normal connected compact subgroup of $G(\R)$ is $C(\R)$, which is isomorphic (as a compact group) to $C/C^{00}$. So by \ref{almexact-isog} we can conclude:

\begin{Corollary}
Assume $G$ has very good reduction. Then 
$G/G^{00}$ is isogenous to the maximal normal connected compact subgroup of $G(\R)$ if and only if
$G$ has the almost exactness property.
\end{Corollary}

\begin{Remark} \label{SU2SL2}
``$G$ has the almost exactness property" does not imply ``$G/G^{00}$ is isomorphic
to $C/C^{00}$". 
\end{Remark}

\begin{proof}
Consider the group $\SU_2(\C)$ of  unitary matrices $2\times 2$ of determinant $1$.
We recall that it is the universal cover of the compact connected simple Lie group $\SO_3(\R)$, with kernel
$\{\pm I\}$. 

Suppose $M =K$ is a saturated real closed field, and $K[i]$ its algebraic closure.
Take $G$ to be the almost direct product $\SU_2(K[i]) \cdot \SL_2(K)$ obtained by identifying the square roots of the identity in both groups. Then $G/G^{00}$ is isomorphic to $\SO_3(\R)$ and $C/C^{00}$ is isomorphic
to $\SU_2(\C)$. 

This is an example of a (semisimple) definable group for which the functor $G \mapsto G^{00}$
is almost exact, but such that $G/G^{00}$ is not isomorphic to $C/C^{00}$. 
\end{proof}

\begin{Proposition} \label{trivialquotient}
Let $G$ be a definably connected group. Then $G/G^{00}$ is isomorphic to $C/C^{00}$ if and only if

\begin{enumerate}
\item $R^{00}$ has finite index in $G^{00} \cap R$.

\item $G/R$ has the exactness property.
\end{enumerate}
\end{Proposition}

\begin{proof}  
We can assume $R$ is definably compact, so $R = Z(G)^0$.

Suppose $G/G^{00}$ is isomorphic to $C/C^{00}$.
By \ref{almexact-isog}, $G$ has the almost exactness property, and in particular $R^{00}$ has finite index in
$G^{00} \cap R$. Set now $G_1 = G/R = G/Z(G)^0$, $C_1 = C/Z(G)^0 = C/Z(C)^0$ and $D_1 = D/Z(G)^0 = D/Z(D)^0$.
Note that $G_1$, $C_1$, $D_1$ are all semisimple (and definably connected), their center is finite, and
$C_1 \cap D_1 \subset Z(G_1)$. Moreover $G_1 = C_1 \cdot D_1$, $C_1$ is definably compact and $D_1$ has no definably compact parts.
In order to show that $G_1$ has the exactness property, by Lemma \ref{reduction-almost-exactness} (see Remark \ref{reduction-exactness}), it is enough to check that $C_1^{00} = G_1^{00} \cap C_1$.
If not, then $C_1 \cap D_1 \subset D_1 = D_1^{00} \subset G_1^{00}$ is not trivial, and
the finite center of $G_1/G_1^{00}$ is not isomorphic to the finite center of $C_1/C_1^{00}$, which is in contradiction with the fact that $G/G^{00}$ is isomorphic to $C/C^{00}$.   
  
Conversely, suppose $R^{00}$ has finite index in $G^{00} \cap R$ and $G/R$ has the exactness property.
By \ref{reduction-almost-exactness} (taking $H = R$), $G$ has the almost exactness property, so $G/G^{00}$ is isogenous to
$C/C^{00}$ (\ref{almexact-isog}). With the same argument used before, we can deduce that $(G^{00} \cap C)/C^{00} \subset (R \cdot C^{00})/C^{00}$, i.e. every element which is in $G^{00} \cap C$ and not in $C^{00}$ belongs to $R$ (otherwise $G/R$ does not have the exactness property). Since compact connected commutative Lie groups which are isogenous are actually isomorphic, we can conclude that $G/G^{00}$ is isomorphic to $C/C^{00}$. 
\end{proof}












\begin{Theorem} \label{strong=defamenable}
Let $G$ be a definably connected group. Then $G$ has the strong exactness property if and only if
$\bar{G}$ is definably compact.
\end{Theorem}

\begin{proof}
If $\bar{G} = G/W$ is definably compact, then we can use the analogue of \ref{reduction-almost-exactness},
(see Remark \ref{reduction-exactness}) setting $H = W$, and deduce that $G$ has the strong exactness property.

Suppose now $G$ has the strong exactness property. It is easy to see that strong exactness is preserved under quotients. So (using the rudimentary decomposition theorem 2.6 of \cite{CCI}) it is enough to show that $D_{2}$ (the semisimple with no definably compact parts, part of $G$) does not have the strong exactness property, unless it is trivial. But as in the last paragraph of section 2, $D_{2}$ (if nontrivial) has a nontrivial definably compact definably connected definable subgroup $C_{2}$. Then $C_{2}^{00}$ is a proper subgroup of $C_{2}$, but $D_{2} = D_{2}^{000}$, which completes the proof.
\end{proof}


Hence using Proposition 4.6 and Corollary 4.12 of \cite{CCI} we can conclude  (where the reader is referred to \cite{CCI} for definitions):

\begin{Corollary}
Let $G$ be a definably connected group.  
Then the following are equivalent:

\begin{enumerate}
\item $G$ is definably amenable.

\item $G$ has a bounded orbit.

\item $G$ has the strong exactness property.

\item $G/W$ is definably compact.
\end{enumerate}

\end{Corollary}

\section{$D^{00}/D^{000}$ and universal covers of semisimple Lie groups}
We recall notation: $D$ is a definably connected central extension of a definable semisimple  (with no definably compact parts) group $D_{2}$ by a definably connected definably compact group $\Gamma$.

In this final section of the paper we will improve slightly on the results from section 3 by proving (with notation there).
\begin{Proposition} 
(i)  $D^{\prime}$ ($= [D,D]$) is perfect, hence equals $(D^{\prime})^{\prime}$). 
\newline
(ii) $D^{\prime} \cap \Gamma$ is finitely generated.
\newline
(iii) $D^{00}/D^{000}$ is (naturally) isomorphic to the quotient of a connected compact commutative Lie group by a finitely generated dense subgroup.
\end{Proposition} 

The ``new ingredient" compared with section 3, is the following, coming from \cite{HPP}:
\begin{Lemma} The structure $(D,\cdot,\Gamma)$ consisting of $D$ its group operation and a predicate for the subgroup $\Gamma$, is (abstractly) isomorphic to some $(D_{1},\cdot, \Gamma_{1})$ where $D_{1}, \Gamma_{1}$ are definable in $(K,+,\cdot)$ over the real algebraic numbers (and moreover $\Gamma_{1}$ is also definably compact although this will not be needed). 
\end{Lemma}
\begin{proof} This is contained in Theorem 6.1(2) of \cite{HPP} and its proof.
\end{proof}

We now work towards the proof of Proposition 5.1.
Let $D_{1}/\Gamma_{1} = D_{3}$, and $\pi:D_{1} \to D_{3}$ the canonical surjective homomorphism. So $D_{3}$ is a semialgebraic (semialgebraically connected) semisimple group in $K$ defined over $\R$. Then passing to real points, $D_{1}(\R)$ is a connected Lie group, $\Gamma_{1}(\R)$ is a connected closed subgroup, and $D_{3}(\R)$ is a connected semisimple Lie group, moreover they are all semialgebraic.   Let $\pi(\R)$ denote the surjective homomorphism from $D_{1}(\R)$ to $D_{3}(\R)$ induced by $\pi$ (with kernel $\Gamma_{1}(\R)$).  Let $u: \widetilde{D_{3}(\R)} \to D_{3}(\R)$ be the universal cover of $D_{3}(\R)$ as a topological (or Lie) group. By the universal properties of $u$ (and as $D_{1}(\R)$ is a central extension of $D_{3}(\R)$), there is a unique homomorphism of Lie groups  $f:\widetilde{D_{3}(\R)} \to D_{1}(\R)$ such that  $u = \pi(\R)\circ f$. 

Let $H$ denote $Im(f)$, the image of $\widetilde{D_{3}(\R)}$ under $f$.
With this notation:

\begin{Lemma}  $H$ is perfect.

\end{Lemma}
\begin{proof}  $\widetilde{D_{3}(\R)}$, as the universal cover of a connected semisimple Lie group, is known to be perfect. Hence so is its homomorphic image  $H$.
\end{proof}

\begin{Lemma}  $H = D_{1}(\R)^{\prime}$  (the commutator subgroup of $D_{1}(\R)$).
\end{Lemma}
\begin{proof} As $H$ maps onto  $D_{3}(\R)$ under $\pi(\R)$, $D_{1}(\R) = \Gamma_{1}(\R)\cdot H$. 

So $[D_{1}(\R),D_{1}(\R)] \subseteq H$, so by the lemma above, we get equality.
\end{proof}

Let $\Lambda$ be $ker(u)$  (where remember $u: \widetilde{D_{3}(\R)} \to D_{3}(\R)$ is the universal covering). So $\Lambda$ is the fundamental group of $D_{3}(\R)$ and as such is a finitely generated commutative group. Clearly
$H\cap \Gamma_{1}(\R)$ is precisely  $f(\Lambda)$ so is also a finitely generated commutative group, which we call  $\Lambda_{0}$.

\vspace{5mm}
\noindent
{\em Proof of (i) and (ii) of Proposition 5.1}  We will, to start off with, work with our groups $D_{1},\Gamma_{1},D_{3}$ which are definable (over $\R)$) in $(K,+,\cdot)$. By Lemma 1.1, for each $n$, 
\newline
$[D_{1},D_{1}]_{n} \cap \Gamma_{1}$ = $[D_{1}(\R),D_{1}(\R)]_{n}\cap \Gamma_{1}(\R)$ is a finite set $X_{n}$ say. 
\newline
(I) $\Lambda_{0} = D_{1}(\R)^{\prime} \cap \Gamma_{1}(\R)$ is thus equal to  $D_{1}^{\prime}\cap \Gamma_{1}$. 

The perfectness of $D_{1}(\R)^{\prime}$  (Lemmas 5.3 and 5.4) together with the fact that some 
$[D_{1},D_{1}]_{n}$ projects on to $D_{3}$, proves that 
\newline
(II) $D_{1}^{\prime}$ is perfect. 

\vspace{2mm}
\noindent
By Lemma 5.2, and (II), (I), $D^{\prime}$ is perfect and its intersection with $\Gamma$ is finitely generated, in fact is precisely $\Lambda_{0}$.  This completes the proof of (i) and (ii).

\vspace{5mm}
\noindent
{\em Proof of (iii) of Proposition 5.1.} With our current notation, and Theorem  3.1 and Corollary 3.2, $D^{00}/D^{000}$ is isomorphic to $E/(\Gamma^{00}\cdot \Lambda_{0})$ where $E$ is type-definable, contains $\Gamma^{00}$ and $E/\Gamma^{00}$ is the closure of $\Gamma^{00}\cdot \Lambda_{0}$ in the compact connected commutative Lie group $\Gamma/\Gamma^{00}$. Let $a_{1},..,a_{k}$ be a finite set of generators of  $(\Gamma^{00}\cdot\Lambda_{0})/\Gamma^{00}$. 
Each $a_{i}$ is a member of some closed connected $1$-dimensional subgroup $A_{i}$ of $\Gamma/\Gamma_{0}$  (by the structure of compact connected Lie groups). Now for each $i$, either the subgroup $\langle a_{i} \rangle$ generated by $a_{i}$ is infinite (cyclic) in which case it is dense in $A_{i}$, OR
$\langle a_{i} \rangle$ is finite (i.e. $a_{i}$ has finite order). Let $A$ be the subgroup of $\Gamma/\Gamma^{00}$ generated by the $A_{i}$ for which $\langle a_{i}\rangle$ is infinite, and $B$ be the  subgroup of $\Gamma/\Gamma^{00}$ generated by the $a_{i}$ of finite order. Then $A$ is a closed connected subgroup, $B$ is finite, and clearly $E/\Gamma^{00} = A\cdot B$. Hence  $(E/\Gamma^{00})/(\Gamma^{00}\cdot\Lambda_{0}/\Gamma^{00})$ is  isomorphic to 
the quotient of the connected group $A$ by its dense subgroup generated by the relevant $a_{i}$. 

\vspace{5mm}
\noindent
Let us summarise the relationship between $G^{00}/G^{000}$ and universal covers. Let us fix $G$ and let $D$ be as in Proposition 1.2, and $1 \to \Gamma \to D \to D_{2}$ as at the beginning of this section. We know that $D_{2}$ has ``very good reduction" so it makes sense to speak about the semisimple (semialgebraic) real Lie group $D_{2}(\R)$. Then:
\begin{Remark}  $G^{00}/G^{000}$ is ``naturally" of the form $A/\Lambda_{00}$ where $A$ is a commutative compact Lie group, and $\Lambda_{00}$ is a dense subgroup and is also  a quotient of the fundamental group of $D_{2}(\R)$. 
\end{Remark}

We also have:
\begin{Remark} For any dense finitely generated subgroup $\Lambda_{00}$ of a connected commutative compact Lie group $A$ there is a semialgebraic group $D$ in a saturated real closed field $(K,+,\cdot)$ such that $D^{00}/D^{000}$ is isomorphic to $A/\Lambda_{00}$. 
\end{Remark}
\noindent
{\em Brief explanation.}  This follows as in  Example 2.10 and Theorem 3.4 from \cite{CCI}, using the fact that for some finite product $G$ of copies of $\SL_{2}(\R)$, $\Lambda_{0}$ is the kernel of some covering homomorphism $H\to G$.


\end{document}